\documentclass[a4paper,11pt,twoside,reqno]{amsart}

\usepackage[utf8]{inputenc}
\usepackage[plainpages=false,pdfpagelabels=true]{hyperref}
\usepackage{amssymb,amsthm}
\usepackage[margin=1in]{geometry}
\usepackage{comment}
\usepackage{tikz}
\usepackage{dsfont}
\usepackage[foot]{amsaddr}

\DeclareMathOperator{\sech}{sech}

\newtheorem{theorem}{Theorem}[section]
\newtheorem{proposition}[theorem]{Proposition}
\newtheorem{lemma}[theorem]{Lemma}

\theoremstyle{definition}

\newtheorem{remark}[theorem]{Remark}

\newtheorem*{thma}{Theorem A}
\newtheorem*{thmb}{Theorem B}
\newtheorem*{thmc}{Theorem C}
\newtheorem*{thmd}{Theorem D}

\newcommand{\tr}{\operatorname{Tr}}

\newcommand{\C}{\mathbb{C}}

\newcommand{\dv}{\text{ }dV}

\newcommand{\N}{\mathbb{N}}
\newcommand{\Z}{\mathbb{Z}}
\newcommand{\R}{\mathbb{R}}

\newcommand{\nor}{\mathrm{nor}}

\newcommand{\SU}{\mathrm{SU}}
\newcommand{\U}{\mathrm{U}}
\newcommand{\CPn}{\mathbb{CP}^{n}}

\DeclareMathOperator{\codim}{codim}
\DeclareMathOperator{\ind}{index}
\DeclareMathOperator{\nullity}{nullity}
\DeclareMathOperator{\diag}{diag}

\allowdisplaybreaks[1]

\renewcommand{\epsilon}{\varepsilon}

\numberwithin{equation}{section}

\title{Explicit harmonic self-maps of complex projective spaces}
\author{Jos\'e Miguel Balado-Alves}
\address{WWU M\"unster, Mathematisches Institut\\
Einsteinstr. 62\\
48149 M\" unster\\
Germany}
\email{jose.balado@uni-muenster.de}

\date{\today}

\thanks{Funded by the Deutsche Forschungsgemeinschaft (DFG, German Research Foundation) under Germany’s Excellence Strategy EXC 2044–390685587, Mathematics Münster: Dynamics–Geometry–Structure and the CRC 1442 Geometry: Deformations and Rigidity}

\begin{document}

\begin{abstract}
We study $ { \SU ( p + 1 ) \times \SU ( n - p ) } $-equivariant maps between complex projective spaces. For every $ { n, p \in \N } $ with $ { 0 \leq p < n } $, we construct two explicit families of uncountable many harmonic self-maps of $ \CPn $, one given by holomorphic maps and the other by maps that are neither holomorphic nor antiholomorphic. We prove that each solution is equivariantly weakly stable and explicitly compute the equivariant spectrum for some specific maps in both families.
\end{abstract}

\maketitle
\section{Introduction}\label{sec1}

Let $ \psi : ( M, g_1 ) \to ( N, g_2 ) $ be a smooth map between two closed Riemannian manifolds, the energy functional of $ \psi $ is given by
\begin{equation*}\small
    E ( \psi ) = \frac{ 1 }{ 2 } \int_M | d \psi |^2 \, \dv_{ g_1 },
\end{equation*}
where $ \dv_{ g_1 } $ denotes the Riemannian volume form of $ M $ with respect to $ g_1 $. A smooth map is called \emph{harmonic} if it is a critical point of the energy functional, i.e., satisfies the Euler-Lagrange equation
\begin{equation}\label{harmonicmapequation}\small
    \tau ( \psi ) := \tr \nabla d \psi = 0 .
\end{equation}
The section $ \tau ( \psi ) \in \Gamma ( \psi^* TN ) $ is called the \emph{tension field} of $ \psi $.

\smallskip
The harmonic map equation~\eqref{harmonicmapequation} is a second-order semilinear elliptic partial differential equation, so proving the existence of non-trivial solutions is generally challenging. This problem raised the interest of mathematicians and physicists in the last decades. Probably the most well-known result concerning existence has been proved by  Eells and Sampson in their famous paper~\cite{eells1964harmonic}, which states that if the target manifold $N$ has non-positive sectional curvature, then there exists a harmonic representative in every homotopy class. 

\smallskip
When the target manifold has positive curvature, the task seems to be much more difficult. One frequently used strategy is to impose some symmetry conditions so that equation~\eqref{harmonicmapequation} reduces to an ordinary differential equation. We refer the reader to the classical book of Eells and Ratto~\cite{eells1993harmonic} for a more thorough discussion of these methods. In this direction, P\"uttmann and Siffert~\cite{puttman2019harmonic} considered equivariant self-maps of compact cohomogeneity one manifolds whose orbit space is a closed interval, and in this setting they reduced the problem of finding harmonic representatives of the homotopy classes to solving singular boundary value problems for nonlinear second-order ordinary differential equations.

\smallskip
Apart from their existence, one of the most relevant questions concerning harmonic maps is their stability.  Roughly speaking, if a given harmonic map is stable, then there does not exist a second harmonic map “nearby”, meaning that the critical points of~\eqref{harmonicmapequation} are isolated. We refer the reader to the book~\cite{urakawa2013calculus} for an introduction to the notion of stability of harmonic maps. Recently, Branding and Siffert~\cite{branding2023equivariant} studied the equivariant stability of equivariant harmonic self-maps of compact cohomogeneity one manifolds, i.e., they consider the stability for variations that are invariant under the cohomogeneity one action.

\smallskip
In this manuscript we consider the complex projective space $ \CPn $ equipped with the Fubini-Study metric and the cohomogeneity one action of $ { \SU ( p + 1 ) \times \SU ( n - p ) } $ acting in the natural way. We then restrict ourselves to maps of the form
\begin{equation*}\small
    \psi : g \cdot \gamma ( t ) \mapsto g \cdot \gamma ( r ( t ) ),
\end{equation*}
where $ g \in { \SU ( p + 1 ) \times \SU ( n - p ) } $ and $ r : [ 0, \tfrac{ \pi }{ 2 } ] \to \R $ is a smooth function satisfying $ r ( 0 ) = 0 $ and $ r ( \tfrac{ \pi }{ 2 } ) = k \tfrac{ \pi }{ 2 } $, which are equivariant with respect to the above action. We show that in this case, the harmonicity of $ \psi $ reduces to the singular boundary value problem:
\begin{equation}\label{odeintroduction}\small
    \begin{split}
        \ddot{ r } ( t ) &+ [ ( 2 n - 2 p - 1 ) \cot t - ( 2 p + 1 ) \tan t ] \, \dot{ r } ( t ) \\[1ex]
        &+ \left[ \frac{ p }{ \cos^2 t } - \frac{ ( n - p - 1 ) }{ \sin^2 t } \right] \, \sin 2 r ( t ) - \frac{ \sin 4 r ( t ) }{ \sin^2 2 t } = 0
    \end{split}
\end{equation}
where
\begin{equation*}
    \lim_{ t \rightarrow 0 } r ( t ) = 0 \quad \text{and} \quad \lim_{ t \rightarrow \frac{ \pi }{ 2 } } r ( t ) = k \tfrac{ \pi }{ 2 }
\end{equation*}
and $ k \in 2 \Z + 1$.

\smallskip
The main results we present here are the following.
\begin{thma}\label{thma}
    Let $ n \in \N $ be given and fix $ p \in \N $ such that $ 0 \leq p < n $. There exists a family of harmonic $ { \SU ( p + 1 ) \times \SU ( n - p ) } $-equivariant maps $ \psi_{ \rho } $ for $ \rho \in \R \setminus \{ 0 \} $ such that for every $ \rho > 0 $ the map $ \psi_{ \rho } $ is holomorphic and for every $ \rho < 0 $ the map $ \psi_{ \rho } $ is neither holomorphic nor anti-holomorphic. Moreover, the energy of each of these maps is given by
    \begin{equation*}\small
        E ( \psi_{ \rho } ) = \tfrac{ \pi^n } { ( n - 1 ) ! } .
    \end{equation*}
\end{thma}
We prove this by giving a family of uncountable many explicit solutions $ r_{ \rho } $ to the singular boundary value problem~\eqref{odeintroduction}, which are independent of $ n $ and $ p $. This family of solutions converges against a limiting configuration when the parameter $ \rho $ approaches $ \pm \infty $, as the following theorem shows.

\begin{thmb}\label{thmb}
    Let $ n \in \N $ be given and fix $ p \in \N $ such that $ 0 \leq p < n $. The solutions $ r_{ \rho } $ to the singular value problem~\eqref{odeintroduction} converge uniformly to $ \pm \tfrac{ \pi }{ 2 } $ as $ \rho $ goes to $ \pm \infty $.
\end{thmb}

With respect to the stability of the harmonic maps $ \psi_{ \rho }$, we obtain the following:
\begin{thmc}\label{thmc}
    Let $ n \in \N $ be given and fix $ p \in \N $ such that $ 0 \leq p < n $.
    \begin{enumerate}
        \item[1.] For every $ \rho > 0 $, the harmonic map $ \psi_{ \rho } : \CPn \to \CPn $ is weakly stable.
        \item[2.] For every $ \rho < 0 $, the harmonic map $ \psi_{ \rho } : \CPn \to \CPn $ is equivariantly weakly stable with respect to the $ { \SU ( p + 1 ) \times \SU ( n - p ) } $-action.
    \end{enumerate}
\end{thmc}

In addition, we explicitly compute the equivariant spectra for the maps $ \psi_{ 1 } $ and $ \psi_{ - 1 } $ in some specific cases.
\begin{thmd}\label{thmd}
    For any $ n \in 2 \N + 1 $, the $ { \SU ( \tfrac{ n + 1 }{2} ) \times \SU ( \tfrac{ n + 1 }{2} ) } $-equivariant spectra of the harmonic maps $ \psi_{ 1 } $ and $ \psi_{ - 1 } $ is given by
    \begin{equation*}\small
        \{ \lambda_j = 4 j ( j + n + 2 ) : j \in \N \}.
    \end{equation*}
\end{thmd}

\smallskip
The organization of the document is as follows. In Section~\ref{sectionpreliminaries} we introduce some basic notions on complex geometry, harmonic maps between cohomogeneity one manifolds, and stability of solutions. In Section~\ref{sectionconstruction} we delve on the construction of $ { \SU ( p + 1 ) \times \SU ( n - p ) } $-equivariant self-maps of $ \CPn $. The reduction theorem is discussed in Section~\ref{sectionderivationtensionfield}. In Section~\ref{sectionsolutions} we prove Theorems\,A and~B. Theorems\,C and~D are proved in Section~\ref{sectionstability}.

\section{Preliminaries}\label{sectionpreliminaries}

\subsection{Basic notions on complex geometry}

We introduce here some concepts and notations used throughout the manuscript. We follow the reference~\cite{griffiths2014principles} for general aspects of complex geometry.

An \textit{almost complex structure} on a smooth manifold $ M $ is a $ ( 1, 1 ) $-tensor field $ J $ verifying $ J^2 = -\mathds{ 1 } $ when regarded as a vector bundle isomorphism $ J : TM \to TM $. Let $ z = ( z_1, \ldots, z_n ) $ be a holomorphic coordinate system, if we write $ z_j = x_j + i y_j $, every complex manifold possesses an almost complex structure defined by
\begin{equation*}\small
    J \tfrac{ \partial }{ \partial x_j } = \tfrac{ \partial }{ \partial y_j }, \quad J \tfrac{ \partial }{ \partial y_j } = -\tfrac{ \partial }{ \partial x_j }.
\end{equation*}

A \textit{Hermitian metric} on a complex manifold $ M $ is a Riemannian metric $ g $ invariant by the almost complex structure $ J $, that is
\begin{equation*}\small
    g ( X, Y ) = g ( JX, JY ) \text{ for any vector fields $ X $ and $ Y $}.
\end{equation*}
It is natural to extend the real tangent space $ T_{ \R, p } M $ to the complexified tangent space of $ M $ at $ p $: $ T_{ \C, p } M = T_{ p } M \otimes_{ \R } \C $. In this case
\begin{equation*}\small
    T_{ \C, p } M = \C \{ \tfrac{ \partial }{ \partial x_j }, \tfrac{ \partial }{ \partial y_j } \} = \C \{ \tfrac{ \partial }{ \partial z_j }, \tfrac{ \partial }{ \partial \bar{ z }_j } \},
\end{equation*}
where
\begin{equation*}\small
    \tfrac{ \partial }{ \partial z_j } = \tfrac{ 1 }{ 2 } \left( \tfrac{ \partial }{ \partial x_j } - i \tfrac{ \partial }{ \partial y_j } \right), \quad \tfrac{ \partial }{ \partial \bar{ z }_j } = \tfrac{ 1 }{ 2 } \left( \tfrac{ \partial }{ \partial x_j } + i \tfrac{ \partial }{ \partial y_j } \right)
\end{equation*}
are the so-called Wirtinger derivatives. Bear in mind that this vector space has real dimension $ 4 n $. Moreover, we can also consider the decomposition of the complexified tangent space $ T_{ \C, p } M = \C \{ \tfrac{ \partial }{ \partial z_j } \} \oplus \C \{ \tfrac{ \partial }{ \partial \bar{ z }_j } \}$. Here, $ T_p' M = \C \{ \tfrac{ \partial }{ \partial z_j } \} $ is called the \emph{holomorphic tangent space} of $ M $ at $ p $, and $ T_p'' M = \C \{ \tfrac{ \partial }{ \partial \bar{ z }_j } \} $ the \emph{antiholomorphic tangent space} of $ M $ at $ p $. Note also that the operation of complex conjugation sending $ \tfrac{ \partial }{ \partial z_j }$ to $\tfrac{ \partial }{ \partial \bar{ z }_j } $ is well defined and $ T_p'' M = \overline{ T_p' M } $.
    
Any Hermitian metric can be extended uniquely to a symmetric bilinear $ ( 0, 2 ) $-tensor in the complexified tangent space, denoted by $ h $. This tensor satisfies, for any vector fields $ X, Y, Z $, the following properties:
\begin{enumerate}
    \item[1.] $ h ( \bar{X}, \bar{Y} ) = \overline{ h ( X, Y ) } $,
    \item[2.] $ h( Z, \bar{ Z } ) > 0 $,
    \item[3.] $ h ( X, Y ) = 0 $ if $ X, Y \in T' M $ or $ X, Y \in T'' M $.
\end{enumerate}
If $ h $ is the extension of the Hermitian metric to the complexified tangent space, then we can always recover the Riemannian metric on $ TM $ by $ g = \text{Re}( h ) $.

The \emph{fundamental $ 2 $-form} $ \Phi $ of an almost Hermitian manifold $ M $ with almost complex structure $ J $ and metric $ g $ is defined by
\begin{equation*}\small
    \Phi ( X, Y ) = g ( X, J Y )
\end{equation*}
for any vector fields $ X $ and $ Y $. A Hermitian metric on an almost complex manifold is called a \emph{K\"ahler metric} if the fundamental $ 2 $-form is closed. A complex manifold with a K\"ahler metric is called a \emph{K\"ahler manifold}.

If a $ C^{ \infty } $-map $ \psi : M \to N $ between two complex manifolds $ M, N $ satisfies either
\begin{equation*}\small
    J \circ d \psi = d \psi \circ J, \text{ or } J \circ d \psi = - d \psi \circ J
\end{equation*}
for every $ p \in M $, then we say that $ \psi $ is \emph{holomorphic} or \emph{antiholomorphic}, respectively.

Every (anti)holomorphic map between two K\"ahler manifolds is harmonic, but not every harmonic map is (anti)holomorphic  (see~\cite{eells1983selected}[Corollary~8.15]). Furthermore, we have the following rigidity result, which can be found, for example, in~\cite{eells1983selected}[page~52].

\begin{theorem}\label{regularity}
    If $ \psi_t : M \to N $ is a smooth deformation of a (anti)holomorphic map $ \psi_0 $ through harmonic maps $ \psi_t $, then each $ \psi_t $ is (anti)holomorphic.
\end{theorem}

\smallskip
For the convenience of the reader, we briefly review the geometry of complex projective spaces. Let $ \CPn $ be the set of all one-dimensional complex-linear subspaces of $ \C^{ n + 1 } $. We can identify $ \CPn $ with the orbit space of $ \C^{ n + 1 } \setminus \{ 0 \} $ under the $ \C^* $-action given by $ \lambda \cdot Z = \lambda Z $, where $ \C^* $ represents the multiplicative group of non-zero complex numbers. This action is smooth, free, and proper, so $ \CPn $ is a $ 2 n $-dimensional manifold with a unique smooth structure such that the quotient map
\begin{equation*}\small
    \pi : \C^{ n + 1 } \setminus \{ 0 \} \to \CPn
\end{equation*}
is a smooth submersion. We refer to $ \CPn $ as \textit{complex projective space}.

Identifying $ \C^{ n + 1 } $ with $ \R^{ 2 n + 2 } $ endowed with its Euclidean metric, we can think of the unit sphere $ \mathbb{ S }^{ 2 n + 1 } $ as an embedded submanifold of $ \C^{ n + 1 } \setminus \{ 0 \} $. Let
\begin{equation*}\small
    p = \pi \vert_{ \mathbb{ S }^{ 2n + 1 } } : \mathbb{ S }^{ 2 n + 1 } \to \CPn
\end{equation*}
denote the restriction of the map $ \pi $. This map is a smooth submersion, so $ \CPn $ is a connected and compact space. Moreover, the action of $ \mathbb{ S }^1 $ on $ \mathbb{ S }^{ 2 n + 1 } $ defined by
\begin{equation*}\small
    \mu \cdot ( z_1, \ldots, z_{ n + 1 } ) = ( \mu z_1, \ldots, \mu z_{ n + 1 } )
\end{equation*}
for $ \mu \in \mathbb{ S }^1 $ (viewed as a complex number of norm $ 1 $) and $ z = ( z_1, \ldots, z_{ n + 1 } ) \in \mathbb{ S }^{ 2 n + 1 } $, is isometric, vertical (meaning that each  $ \mu \in \mathbb{ S }^1 $ takes each fiber to itself), and transitive on fibers of $ p $. Therefore, there is a unique metric on $ \CPn $ such that the map $ p : \mathbb{ S }^{ 2 n + 1 } \to \CPn $ is a Riemannian submersion. This metric is called the \textit{Fubini-Study metric}.

Let $ [ Z_0 : \ldots : Z_n ] $ be the standard homogeneous coordinates on $ \CPn $. Consider the chart $ U_0 = \{ Z \in \CPn : Z_0 \neq 0 \} $. Then $ z = ( z_1, \ldots, z_n ) $ is a coordinate system in $ U_0 $, where $ z_j = \tfrac{ Z_j }{ Z_0 } $. In this coordinate system, we can write (up to a positive constant) the Fubini-Study metric, $ g_{ FS } $, as the real part of
\begin{equation}\label{FSmetric} 
\small
    h_{ FS } = \sum_{ j, k = 1 }^n \frac{ ( 1 + | z |^2 ) \, \delta_{ j k } -\bar{ z_j } \, z_k }{ ( 1 + | z |^2 )^2 } \, dz_j \, d\bar{ z }_k.
\end{equation}
$ \CPn $ equipped with the metric $ g_{ FS } $ is a K\"ahler manifold, and the sectional curvature of the plane spanned by two orthonormal vectors $ X, Y \in T_p \CPn $ is
\begin{equation*}\small
    K_p ( X , Y ) = 1 + 3 g_{ FS } ( X , J Y )^2,
\end{equation*}
i.e., for every $ 2 $-plane $ \sigma $ we have $ 1 \leq K_p ( \sigma ) \leq 4 $.

\subsection{Harmonic maps between cohomogeneity one manifolds}

We give in this subsection a brief introduction to harmonic maps between manifolds equipped with a cohomogeneity one action. The main source is~\cite{puttman2019harmonic}.

Let $ ( M, \langle \cdot, \cdot \rangle ) $ be a Riemannian manifold endowed with an isometric action $ G \times M \to M $ of a compact Lie group such that the orbit space $ M / G $ is isometric to a closed interval $ [ 0, L ] $ and such that the Weyl group of the action is finite. The endpoints $ 0 $ and $ L $ correspond to non-principal orbits $ N_0 $ and $ N_1 $ while each interior point corresponds to a principal orbit. The so-called $ ( k, r ) $-maps are maps $ \psi: M \to M $ of the form
\begin{equation*}\small
    g \cdot \gamma ( t ) \mapsto g \cdot \gamma ( r ( t ) )
\end{equation*}
where $ g \in G $ and $ r : [ 0, L ] \to \R $ is a smooth function satisfying $ r ( 0 ) = 0 $, $ r ( L ) = k L $. We denote by $ \gamma $ a fixed unit speed normal geodesic where $ \gamma ( 0 ) \in N_0 $ and $ \gamma ( L ) \in N_1 $. P\"uttmann's result~\cite{puttmann2009cohomogeneity} ensures that the map $ \psi $ is smooth if $ k $ is of the form $ j | W | / 2 + 1 $ where $ j $ is, in general, an even integer and, if the isotropy group at $ \gamma ( L ) $ is a subgroup of the isotropy group at $ \gamma ( ( | W | / 2 + 1 ) L ) $, then $ j $ is also allowed to be an odd integer. 

The Brouwer degree of a $ ( k, r ) $-map is given by:
\begin{equation*}
    \deg \psi = \begin{cases}
            k & \text{if $ \codim N_0 $ and $ \codim N_1 $ are both odd,}\\
            + 1 & \text{otherwise,}
        \end{cases}
\end{equation*}
if $ j $ is even, and by
\begin{equation*}
    \deg \psi = \begin{cases}
            k & \text{if $ \codim N_0 $ and $ \codim N_1 $ are both odd,}\\
            0 & \text{if $ \codim N_0 $ and $ \codim N_1 $ are both even, $ | W | \not \in 4 \Z $,}\\
            -1 & \text{if $ \codim N_0 $ is even, $ \codim N_1 $ is odd, and $ | W | \not \in 8 \Z $,}\\
            +1 & \text{otherwise,}
        \end{cases}
\end{equation*}
if $ j $ is odd.

The principal isotropy groups $ H = G_{ \gamma ( t ) } $ along the normal geodesic are constant for $ 0 < t < L $. We write $ \mathfrak{ g } $, $ \mathfrak{ h } $ for the Lie algebras of $ G $ and $ H $, respectively. Let $ Q $ be a biinvariant metric on $ G $ and denote by $ \mathfrak{ n } $ the orthogonal complement of $ \mathfrak{ h } $ in $ \mathfrak{ g } $. Define the endomorphism $ { P_t : \mathfrak{ n } \to \mathfrak{ n } } $ by
\begin{equation*}\small
    Q ( P_t X, Y ) = \langle X^{ \ast }, Y^{ \ast } \rangle_{ \vert \gamma ( t ) }
\end{equation*}
for every $ X, Y \in \mathfrak{ n } $, where $ X^{ \ast }, Y^{ \ast } $ are their corresponding action fields defined by
\begin{equation*}\small
    X^{ \ast }_{ \vert \gamma ( t ) } = \frac{ d }{ ds } \Big \vert_{ s = 0 } \exp ( s X ) \cdot \gamma ( t ) .
\end{equation*}
If we split the tension field into the normal and tangential components with respect to the principal orbits
\begin{equation*}\small
    \tau_{ \vert \gamma ( t ) } = \tau_{ \vert \gamma ( t ) }^{ \text{nor} } + \tau_{ \vert \gamma ( t ) }^{ \text{tan} },
\end{equation*}
then one can express the tension field of a $ ( k, r ) $-map in terms of the endomorphisms $ P_t $ as
\begin{equation*}\small
    \tau^{ \nor }_{ \vert \gamma ( t ) } = \left[ \ddot r ( t ) + \tfrac{ 1 }{ 2 } \dot{ r } ( t ) \, \tr P_t^{ - 1 } \, \dot{ P }_t - \tfrac{ 1 }{ 2 } \tr P_t^{ - 1 } \, ( \dot{ P } )_{ r ( t ) } \right] \dot{ \gamma } ( r ( t ) )
\end{equation*}
and
\begin{equation*}\small
    \tau^{ \tan }_{ \vert \gamma ( t ) } = \bigl( P_{ r ( t ) }^{ - 1 } \, \sum_{ \mu = 1 }^n \, [ E_{ \mu }, P_{ r ( t ) } E_{ \mu } ] \bigr)^{ \ast }_{ \vert \gamma ( r ( t ) ) },
\end{equation*}
where $ E_1, \ldots, E_n \in \mathfrak{ n } $ are such that $ E^{ \ast }_{ 1 \vert \gamma ( t ) }, \ldots, E^{ \ast }_{ n \vert \gamma ( t ) } $ form an orthonormal basis of $ { T_{ \gamma ( t ) } G \cdot \gamma ( t ) } $ and $ [ \cdot, \cdot ] $ represents the Lie bracket of $ \mathfrak{ g } $. See Theorems\,3.4 and~3.6 in~\cite{puttman2019harmonic}, respectively.

\subsection{Stability of harmonic maps}

We present here some notions about the stability and equivariant stability of harmonic maps. We follow the references~\cite{urakawa2013calculus} and~\cite{branding2023equivariant}.

Let $ \psi : ( M, g_1 ) \to ( N, g_2 ) $ be a harmonic map, the second variation of the energy is given by
\begin{equation*}\small
    \delta^2 E ( \psi ) ( V, W ) = \int_M g_2 ( J_{ \psi } ( V ), W ) \, dV_{g_1}, \quad V, W \in \Gamma ( \psi^* TN ) .
\end{equation*}
Here $ J_{ \psi } $ denotes the Jacobi operator, this is, the second order selfadjoint linear elliptic differential operator defined by
\begin{equation*}\small
    J_{ \psi } ( V ) := - \sum_{ i = 1 }^{ m } \left( \nabla_{ e_i } \nabla_{ e_i } - \nabla_{ \nabla_{ e_i } e_i } \right ) V - \sum_{ i = 1 }^m R^N ( V, d \psi ( e_i ) ) \, d \psi ( e_i )
\end{equation*}
for every $ V \in \Gamma ( \psi^* TN ) $, where $ R^N $ represents the Riemannian curvature tensor of $ N $ and $ \{ e_i \}_{ i = 1 }^{ m } $ is an orthonormal basis of $ TM $.

A harmonic map $ \psi $ is \emph{stable} if
\begin{equation*}\small
    \delta^2 E ( \psi ) ( V, V ) > 0
\end{equation*}
for all $ V \in \Gamma ( \psi^* TN ) $. If $ \delta^2 E ( \psi ) ( V, V ) \geq 0 $, then we say $ \psi $ is \emph{weakly stable}, otherwise we say the map is \emph{unstable}. Due to the general theory of linear elliptic operators on a compact Riemannian manifold, the spectrum of $ J_{ \psi } $ consists only of a discrete set of an infinite number of eigenvalues, denoted as
\begin{equation*}\small
    \lambda_0 ( \psi ) < \lambda_1 ( \psi ) < \ldots < \lambda_j ( \psi ) \to \infty.
\end{equation*}
The vector space
\begin{equation*}\small
    V_{ \lambda }( \psi ) : = \{ V \in \Gamma ( \psi^* TN ) : J_{ \psi } V = \lambda V \} \neq \{ 0 \},
\end{equation*}
is called \emph{eigenspace} with eigenvalue $ \lambda $. We define
\begin{equation*}\small
    \begin{split}
        \ind ( \psi ) &: = \sum_{ \lambda < 0 } \dim V_{ \lambda } ( \psi ), \\[1ex]
        \nullity ( \psi ) &: = \dim V_0 ( \psi ) .
    \end{split}
\end{equation*}
Thus, $ \psi $ is weakly stable if and only if $ \lambda_j ( \psi ) \geq 0 $ for every $ j \in \N $. We refer to~\cite{urakawa2013calculus}[Theorem~3.2] for a proof of the following theorem, which we found useful for our purposes.
\begin{theorem}\label{weaklystability}
    Let $ ( M, g_1 ) $, $ ( N, g_2 ) $ be two compact K\"ahler manifolds and $ \psi : M \to N $ a holomorphic harmonic map. Then the following holds:
    \begin{equation*}\small
        \int_M g_2 ( J_{ \psi } V, V ) \, dV_{g_1} = \frac{ 1 }{ 2 } \int_M g_2 ( D V, D V ) \, dV_{g_1} \geq 0, \quad V \in \Gamma ( \psi^* TN ),
    \end{equation*}
    where $ D V $ is an element of $ \Gamma ( \psi^* T N \otimes T^* M ) $ defined by
    \begin{equation*}\small
        D V ( X ) := \nabla_{ J X } V - J \nabla_X V, \quad X \in \mathfrak{ X } ( M ) .
    \end{equation*}
    In particular, $ \psi $ is weakly stable.
\end{theorem}

\smallskip
Since in our setting we are working with equivariant $ ( k, r ) $-maps, it is natural to study also the \emph{equivariant stability} of the harmonic solutions, in the sense that one considers only variations that are invariant under the cohomogeneity one action. In general, the notion of stability is more restrictive than the notion of equivariant stability: there exist examples of harmonic maps which are unstable but equivariantly stable under some group action. The main source for equivariant stability of harmonic self-maps of cohomogeneity one manifolds is~\cite{branding2023equivariant}.

The spectral problem describing the equivariant stability of a harmonic $ ( k, r ) $-map can also be expressed in terms of the endomorphisms $ P_t $ as
\begin{equation*}\small
    \ddot{ \xi } ( t ) + \tfrac{ 1 }{ 2 } \tr( P_t^{ - 1 } \dot{ P }_t ) \, \dot{ \xi } ( t ) - \tfrac{ 1 }{ 2 } \tr ( P_t^{ - 1 } \ddot{ P }_{ r ( t ) } ) \, \xi ( t ) + \lambda \xi ( t ) = 0
\end{equation*}
where $ \xi \in C_0^{ \infty }( [ 0, \tfrac{ \pi }{ 2 } ] ) $. We will refer to the spectrum for the Sturm-Liouville problem above as \emph{equivariant spectrum}.

The so-called \emph{Jacobi polynomials} proved to be useful in the study of these differential equations. Recall that the Jacobi polynomials $ P^{ ( \alpha, \beta ) }_j ( t ) $, where $ j \in \N $ and $ \alpha, \beta > - 1 $, solve the following differential equation:
\begin{equation}\label{jacobipolynomial}
    ( 1 - t^2 ) \ddot{ \xi } ( t ) + [ \beta - \alpha - ( \alpha + \beta + 2 ) t ] \, \dot{ \xi } ( t ) + j ( j + 1 + \alpha + \beta ) \xi ( t ) = 0.
\end{equation}
See~\cite[\href{https://dlmf.nist.gov/18.3}{(18.3)}]{NIST:DLMF} for more details about the Jacobi and other orthogonal polynomials.

\section{\texorpdfstring{$ \SU ( p + 1 ) \times \SU ( n - p ) $}{SU(p+1) x SU(n-p)} - equivariant self-maps of \texorpdfstring{$ \CPn $}{CP}}\label{sectionconstruction}
Let $ n \in \N $ be given and fix $ p \in \N $ such that $ 0 \leq p < n $. Consider the compact Riemannian manifold $ ( \CPn, g_{ FS } ) $. The group $ { G = \SU ( p + 1 ) \times \SU ( n - p ) } $ is compact, and its natural action on the unit sphere of $ \C^{ n + 1 } $ is isometric. Thus, the action $ G \times \CPn \to \CPn $ given by
\begin{equation*}\small
    \left(
    \begin{pmatrix}
        A & 0 \\
        0 & B 
    \end{pmatrix}
    , [ Z ] \right) \to \left[
    \begin{pmatrix}
        A & 0 \\
        0 & B 
    \end{pmatrix}
    Z \right] ,
\end{equation*}
where $ Z $ is in homogeneous coordinates, $ A \in \SU ( p + 1 ) $, and $ B \in \SU ( n - p ) $, is also isometric. Uchida proved in~\cite{uchida1977classification} that this is a cohomogeneity one action with orbit space $ \CPn / G $ isometric to the closed interval $ { [ 0, \tfrac{ \pi }{ 2 } ] } $.

It is not hard to see that the geodesic $ \gamma $ defined by
\begin{equation*}\small
    \gamma ( t ) = [ \cos t \, e_1 +  \sin t \, e_{p+2} ]
\end{equation*}
for every $ t \in [ 0, \tfrac{ \pi }{ 2 } ]$, where $ \{ e_j \}_{ j = 1 }^{ n + 1 } $ is the standard basis of $ \R^{ n + 1 } $, is a normal geodesic, i.e., a unit-speed geodesic that passes through all orbits perpendicularly.

A direct computation shows that the isotropy groups are given by

\begin{equation*}\small
    K_0 = \left \{ \begin{pmatrix}
        \nu  & 0 & 0  \\
        0 & A & 0 \\
        0 & 0 & B
    \end{pmatrix} : A \in \U ( p ), \, B \in \SU ( n - p ), \, \nu^{ -1 } = \det A  \right \},
\end{equation*}
    
\begin{equation*}\small
    H = \left\{ \begin{pmatrix}
        \nu  & 0 & 0 & 0 \\
        0 & A & 0 & 0\\
        0 & 0 & \nu & 0 \\
        0 & 0 & 0 & B
    \end{pmatrix} : A \in \U ( p ), \, B \in \U ( n - p - 1 ), \, \nu^{ -1 }= \det A =  \det B   \right \},
\end{equation*}
    
\begin{equation*}\small
    K_1 = \left \{ \begin{pmatrix}
        A  & 0 & 0  \\
        0 & \nu & 0 \\
        0 & 0 & B
    \end{pmatrix} : A \in \SU ( p + 1 ), \, B \in \U ( n - p - 1 ), \, \nu^{ -1 } = \det B  \right\},
\end{equation*}
where we write $ H = G_{ \gamma ( t ) } $ for every $ t \in ( 0, \tfrac{ \pi }{ 2 } ) $, $ K_0 = G_{ \gamma ( 0 ) } $, and $ K_1 = G_{ \gamma( \frac{ \pi }{ 2 } ) }$.

Recall that the Weyl group $ W $ is the dihedral subgroup of $ N ( H ) / H $ generated by the two unique involutions
\begin{equation*}\small
    \sigma_{j} \in \frac{ N ( H ) \cap K_j }{ H },
\end{equation*}
$ j =  0, 1 $, and the non-principal isotropy groups along $ \gamma ( \R ) $ are conjugate to one of the $ K_j $ via an element of $ W $.

In the present setting, the Weyl group is a dihedral group of order $ | W | = 2 $ generated by the two involutions
\begin{equation*}\small
    \sigma_0 = \begin{pmatrix}
        \mathds{ 1 }_{ p + 1 } & 0 & 0 & 0 \\
        0 & -1 & 0 & 0\\
        0 & 0 & -1 & 0 \\
        0 & 0 & 0 & \mathds{ 1 }_{ n - p - 2 }
    \end{pmatrix}, \quad 
    \sigma_1 = \begin{pmatrix}
        -1 & 0 & 0 & 0 \\
        0 & -1 & 0 & 0\\
        0 & 0 & \mathds{ 1 }_{ p - 1 } & 0 \\
        0 & 0 & 0 & \mathds{ 1 }_{ n - p }
    \end{pmatrix}.
\end{equation*}
The extended group diagram is of the form presented in Figure~\ref{figureextendedgroupdiagram}.
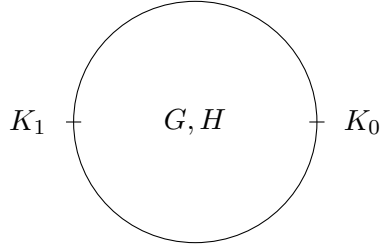
\begin{figure}[ht]
\centering
    \begin{tikzpicture}
        \node at (0,0) {$ G, H $};
        \draw (0,0) circle [radius=1.6cm]; 
        \draw(0:1.5cm)--(0:1.7cm);
        \node at (0:2.2cm) {$ K_0 $};
        \draw(180:1.5cm)--(180:1.7cm);
        \node at (180:2.2cm) {$ K_1 $};
    \end{tikzpicture}
\caption{Extended group diagram for the action of \texorpdfstring{$ \SU ( p + 1 ) \times \SU ( n - p ) $}{SU(p+1) x SU(n-p)} on \texorpdfstring{$ \CPn $}{CP}}
\label{figureextendedgroupdiagram}
\end{figure}

We consider equivariant self-maps of $ \CPn $ under the action of $ G = \SU ( p + 1 ) \times \SU ( n - p ) $ of the form
\begin{equation*}\small
    g \cdot \gamma ( t ) \mapsto g \cdot \gamma ( r ( t ) )
\end{equation*}
where $ g \in G $. Here $ r : [ 0, \tfrac{ \pi }{ 2 } ] \to \R $ is a smooth function satisfying $ r ( 0 ) = 0 $ and $r ( \tfrac{ \pi }{ 2 } ) = k \tfrac{ \pi }{ 2 } $. In this case $ k $ is of the form $ j + 1 $ with $ j \in 2\Z $, since $ G_{ \gamma ( \frac{ \pi }{ 2 } ) } = K_1 $ is not a subgroup of $ G_{ \gamma ( \pi ) } = K_0 $, see~\cite{puttmann2009cohomogeneity}[Lemma 2.1]. 

The codimensions of the singular orbits are both even: $ 2 ( n - p ) $ and $ 2 ( p + 1 ) $, so the Brouwer degree of such maps is $ + 1 $, see~\cite{puttmann2009cohomogeneity}[Theorem 3.4].

\section{Derivation of the tension field}\label{sectionderivationtensionfield}

We write $ \mathfrak{ g }, \mathfrak{ h } $ for the Lie algebra of $ G $ and $ H $, respectively. We choose on $ \mathfrak{ g } $ the inner product $ Q $ defined by
\begin{equation*}\small
    Q ( X, Y ) = - \frac{ 1 }{ 2 } \tr X Y
\end{equation*}
for every $ X, Y \in \mathfrak{ g } $. We denote the orthogonal complement of $ \mathfrak{ h } $ under $ Q $ by $ \mathfrak{ n } $. Define for every $ t \notin \tfrac{ \pi }{ 2 } \Z $ the endomorphisms $ P_t : \mathfrak{ n } \to \mathfrak{ n } $ by
\begin{equation*}\small
    Q ( P_t X, Y ) = g_{ FS } ( X^{ \ast }, Y^{ \ast } )_{ \vert \gamma ( t ) }
\end{equation*}
for every $ X, Y \in \mathfrak{ n } $ where $ X^{ \ast }, Y^{ \ast } $ are their corresponding action fields given by
\begin{equation*}\small
    X^{ \ast }_{ \vert \gamma ( t ) } = \frac{ d }{ ds } \Big\vert_{ s = 0 } \exp ( s X ) \cdot \gamma ( t ) .
\end{equation*}

\subsection{Computation of \texorpdfstring{$ P_t $}{Pt}}

We start by explicitly computing the endomorphism $ P_t $. For that, note that in any basis $ \{ X_1, \ldots, X_{ 2 n - 1 } \}$ of $ ( \mathfrak{ n }, Q ) $ we can compute the action of $ P_t $ on any $ X_j $ as
\begin{equation}\label{formulaPt}\small
    P_t X_j = \sum_{ k = 1 }^{ 2 n - 1 } Q ( P_t X_j, X_k ) X_k = \sum_{ k = 1 }^{ 2 n - 1 } g_{ FS }( X_j^{ \ast }, X_k^{ \ast } )_{ \vert \gamma ( t ) } X_k.
\end{equation}

Throughout this section, $ C_{ j, k } $ represents the $ ( n + 1 ) \times ( n + 1 ) $-matrix with $ 1 $ in the $ ( j, k ) $-entry and $ 0 $ in the others. In addition, we define
\begin{equation*}\small
    E_{ j, k } = C_{ j, k } - C_{ k, j }, \quad F_{ j, k } = i C_{ j, k } + i C_{ k, j },
\end{equation*}
and the diagonal matrix
\begin{equation*}\small
    D = \omega i \, \diag \Big[ p ( n - p ), \, \underbrace{ p - n, \ldots, p - n }_{ \text{$ p $ times} }, \, - ( p + 1 ) ( n - p - 1 ), \, \underbrace{ p + 1, \ldots, p + 1 }_{ \text{$ n - p - 1 $ times} } \Big]
\end{equation*}
where
\begin{equation*}\small
    \omega^2 =  2 ( p + 1 )^{ - 1 } ( n - p )^{ - 1 } ( ( p + 1 ) ( n - p - 1 ) + ( n - p ) p )^{ - 1 } .
\end{equation*}
Consider also the sets
\begin{align*}
\small
    N_1 &= \{ E_{ 1, j } : 2 \leq j \leq p + 1 \} , \,
    &&N_2 = \{ F_{ 1, j } : 2 \leq j \leq p + 1 \} , \\ 
    N_3 &= \{ E_{ p+2, j } : p + 3 \leq j \leq n + 1 \} , \,
    &&N_4 = \{ F_{ p+2, j } : p + 3 \leq j \leq n + 1 \}.
\end{align*}
Then
\begin{equation}\small\label{basis}
    \Lambda = N_1 \cup N_2 \cup N_3 \cup N_4 \cup D
\end{equation}
is an orthonormal basis of $ ( \mathfrak{ n }, Q ) $. We denote by $ \{ e_j \}_{ j = 1 }^{ n + 1 } $ the standard basis of $ \R^{ n + 1 } $ and by $ \{ \xi_j \}_{ j = 1 }^n$ the standard basis of $ \R^n $, we also write $ z = x + iy $ and
\begin{equation*}\small
        \tfrac{ \partial }{ \partial z_j } = \tfrac{ 1 }{ 2 } \left( \tfrac{ \partial }{ \partial x_j } - i \tfrac{ \partial }{ \partial y_j } \right), \quad \tfrac{ \partial }{ \partial \bar{ z }_j } = \tfrac{ 1 }{ 2 } \left( \tfrac{ \partial }{ \partial x_j } + i \tfrac{ \partial }{ \partial y_j } \right).
\end{equation*}

\begin{lemma}\label{actionfieldsofbasis}
For any $ t \in ( 0, \tfrac{ \pi }{ 2 } ) $ and every element of $ \Lambda $ we have that its corresponding action field is given by
\begin{equation*}\small
    \begin{split}
        E_{ j, k \vert \gamma ( t ) }^{ \ast } =& \begin{cases}
            - \tfrac{ \partial }{ \partial z_{ k - 1 } } - \tfrac{ \partial }{ \partial \bar{ z }_{ k - 1 } } & \text{if $ j = 1$, $k=2, \ldots, p + 1 $},\\
            - \tan t \, \left( \tfrac{ \partial }{ \partial z_{ k - 1 } } + \tfrac{ \partial }{ \partial \bar{ z }_{ k - 1 } } \right) & \text{if $ j = p + 2$, $k=p+3, \ldots, n $,}
        \end{cases} \\[2ex]
        F_{ j, k \vert \gamma ( t ) }^{ \ast } =& \begin{cases}
            i \tfrac{ \partial }{ \partial z_{ k - 1 } } - i \tfrac{ \partial }{ \partial \bar{ z }_{ k - 1 } } & \text{if $ j = 1$, $k=2, \ldots, p + 1 $},\\
            \tan t \, \left( i \tfrac{ \partial }{ \partial z_{ k - 1 } } - i \tfrac{ \partial }{ \partial \bar{ z }_{ k - 1 } } \right) \,  & \text{if $ j = p + 2$, $k=p+3, \ldots, n $,}
        \end{cases}
        \\[2ex]
         \hspace{3cm} D_{ \vert \gamma ( t ) }^{ \ast } =& - \eta \tan t \, \left( i \tfrac{ \partial }{ \partial z_{ p + 1 } } - i \tfrac{ \partial }{ \partial \bar{ z }_{ p + 1 } } \right),
    \end{split}
\end{equation*}
where 
\begin{equation*}\small
    \eta^2 = 2 \frac{ n - p - 1 }{ n - p } + 2 \frac{ p }{ p + 1 }.
\end{equation*}
\end{lemma}
\begin{proof}
    For $ E_{ j, k }, F_{ j, k }, D \in \Lambda $, the smooth homomorphisms defined by
    \begin{equation*}\small
    \begin{split}
        \exp s E_{ j, k } &= \sin s \, E_{ j, k } + \cos s \, ( C_{ j, j } + C_{ k, k } ) + \sum_{ \ell \neq j, k } C_{ \ell, \ell }, \\
        \exp s F_{ j, k } &= \sin s \, F_{ j, k } + \cos s \, ( C_{ j, j } + C_{ k, k } ) + \sum_{ \ell \neq j, k } C_{ \ell, \ell }, \\
        \exp s D &= \sum_{ \ell } e^{ s [ D ]_{ \ell, \ell } } C_{ \ell, \ell },
    \end{split}
    \end{equation*}
    for every $ s \in \R $ are one-parameter subgroups of $ G $ generated by the elements of $ \Lambda $. Here, $ [ D ]_{ \ell , \ell } $ denotes the $ ( \ell , \ell ) $-entry of the matrix $ D $. The actions of these one-parameter subgroups at the point $ \gamma ( t ) $ yield curves defined by
    \begin{equation*}\small
        \begin{split}
            \exp s E_{ j, k } \cdot \gamma ( t ) = \, &[ \cos s \, \cos t \, e_1 + \sin t \, e_{ p + 2 } - \, \sin s \, \cos t \, e_k ] \, \delta_{ 1, j } \\
            + \, &[ \cos s \, \sin t \, e_{ p + 2 } + \cos t \, e_{ 1 } - \sin s \, \sin t \, e_k ] \, \delta_{ p + 2, j } \\[1ex]
            \exp s F_{ j, k } \cdot \gamma ( t ) = \, &[ \, \cos s \, \cos t \, e_1 + \sin t \, e_{ p + 2 }+ i \sin s \, \cos t \, e_k ] \, \delta_{ 1, j } \\
            + \, &[ \cos s \, \sin t \, e_{ p + 2 } + \cos t \, e_{ 1 } + i \sin s \, \sin t \, e_k ] \, \delta_{ p + 2, j } \\[1ex]
             \exp s D \cdot \gamma ( t ) = \, &[ \cos t \, e^{ s [ D ]_{ 1, 1 } } \, e_1 + \sin t \, e^{ s [ D ]_{ p + 2, p + 2 } } \, e_{ p + 2 } ] .
        \end{split} 
    \end{equation*}
    Since we are interested in taking the derivative of these curves at $ s = 0 $, we can assume that $ s \in ( - \varepsilon , \varepsilon ) $ with $ \varepsilon $ small enough so that $ \cos s  $ does not vanish there. Moreover, $ t \in ( - \frac{ \pi }{ 2 } , \frac{ \pi }{ 2 } ) $, so there is no restriction on using affine coordinates here, obtaining
    \begin{equation*}\small
        \begin{split}
             \exp s E_{ j, k } \cdot \gamma ( t ) &= ( - \tan s \, \xi_{ k - 1 } + \sec s \, \tan t \, \xi_{ p + 1 } ) \, \delta_{ 1, j } \\
            &+ ( - \sin s \, \tan t \, \xi_{ k - 1 } + \cos s \, \tan t \, \xi_{ p + 1 } ) \, \delta_{ p + 2, j } \\[1ex]
            \exp s F_{ j, k } \cdot \gamma ( t ) &= ( i \tan s \, \xi_{ k - 1 } + \sec s \, \tan t \, \xi_{ p + 1 } ) \, \delta_{ 1, j } \\
            &+ ( i \sin s \, \tan t \, \xi_{ k - 1 } + \cos s \, \tan t \, \xi_{ p + 1 } ) \, \delta_{ p + 2, j }, \\[1ex]
            \exp s D \cdot \gamma ( t ) &= \tan t \, e^{ - s \eta } \, \xi_{ p + 1 }.
        \end{split} 
    \end{equation*}
    Taking the derivative at $ s = 0 $, we obtain the following tangent vectors to $ \CPn $ at $ \gamma ( t ) $:
    \begin{equation*}\small
        \begin{split}
            E_{ j, k \vert \gamma ( t ) }^{ \ast } &= - \delta_{ 1, j } \, \tfrac{ \partial }{ \partial x_{ k - 1 } } - \delta_{ p + 2, j } \, \tan t \, \tfrac{ \partial }{ \partial x_{ k - 1 } }, \\[1ex]
            F_{ j, k \vert \gamma ( t ) }^{ \ast } &= \delta_{ 1, j } \, \tfrac{ \partial }{ \partial y_{ k - 1 } } + \delta_{ p + 2, j } \, \tan t \,  \tfrac{ \partial }{ \partial y_{ k - 1 } }, \\[1ex]
            D_{ \vert \gamma ( t ) }^{ \ast } &= - \eta \tan t \, \tfrac{ \partial }{ \partial y_{ p + 1 } }.
        \end{split}
    \end{equation*}
    The result follows after making the substitutions
    \begin{equation*}\small
        \tfrac{ \partial }{ \partial x_{ k } } = \tfrac{ \partial }{ \partial z_{ k } } + \tfrac{ \partial }{ \partial \bar{ z }_{ k } }, \quad \tfrac{ \partial }{ \partial y_{ k } } = i \tfrac{ \partial }{ \partial z_{ k } } - i \tfrac{ \partial }{ \partial \bar{ z }_{ k } } .
    \end{equation*}
\end{proof}

Following~\eqref{FSmetric}, the Fubini-Study metric at a point $ \gamma ( t ) $ for $ t \neq \tfrac{ \pi }{ 2 } $ can be expressed, in affine coordinates, as
\begin{equation}\label{FSmetricatgeodesic}\small
   ( g_{ FS } )_{ j, k } = \frac{ \cos^2 t }{ 2 } \, 
    \begin{pmatrix}
        \begin{array}{c|c}
            0 & \begin{matrix}
                \mathds{ 1 }_{ p } & {} & {} \\
                {} & \cos^2 t & {} \\
                {} & {} & \mathds{ 1 }_{ n - p - 1 }
            \end{matrix} \\
        \hline
            \begin{matrix}
                \mathds{ 1 }_{ p } & {} & {} \\
                {} & \cos^2 t & {} \\
                {} & {} & \mathds{ 1 }_{ n - p - 1 }
            \end{matrix} & 0
        \end{array}
    \end{pmatrix}
\end{equation}
with respect to the basis $ \{ \tfrac{ \partial }{ \partial z_1 }, \ldots, \tfrac{ \partial }{ \partial z_n }, \tfrac{ \partial }{ \partial \bar{ z }_1 }, \ldots, \tfrac{ \partial }{ \partial \bar{ z }_n } \} $. The next result then follows from a straightforward computation.

\begin{proposition}\label{Pt}
    For $ t \in ( 0, \tfrac{ \pi }{ 2 } ) $, the endomorphism $ P_t : \mathfrak{ n } \to \mathfrak{ n } $ is given by
    \begin{equation*}\small
        P_t = \begin{pmatrix}
            \cos^2 t \, \mathds{ 1 }_{ 2 p } & {} & {} \\
            {} & \sin^2 t \, \mathds{ 1 }_{ 2 ( n - p - 1 ) } & {} \\
            {} &  {} & \tfrac{ \eta^2 }{ 4 } \sin^2 2 t
        \end{pmatrix}
    \end{equation*}
with respect to the basis $ \Lambda $ defined in~\eqref{basis}.
\end{proposition}

\subsection{Tangential component of the tension field}
One can express the tangential component of the tension field of a $ ( k, r ) $-map in terms of the endomorphisms $ P_t $, as the following theorem shows.
\begin{theorem}[see~\cite{puttman2019harmonic}]\label{tangentialpart}
The tangential component of the tension field is given by
\begin{equation*}\small
    \tau^{ \tan }_{ \vert \gamma ( t ) } = \bigl( P_{ r ( t ) }^{ - 1 } \sum_{ \mu = 1 }^n [ E_{ \mu }, P_{ r ( t ) } E_{ \mu } ] \bigr)^{ \ast }_{ \vert \gamma ( r ( t ) ) }
\end{equation*}
where $ E_1, \ldots, E_n \in \mathfrak{ n } $ are such that $ E^{ \ast }_{ 1 \vert \gamma ( t ) }, \ldots, E^{ \ast }_{ n \vert \gamma ( t ) } $ form an orthonormal basis of $ { T_{ \gamma ( t ) } G \cdot \gamma ( t ) } $.
\end{theorem}
 Following the construction of Lemma~\ref{actionfieldsofbasis}, for every $ k = 2, \ldots, p + 1 $ and every $ \ell = p + 3, \ldots, n + 1 $, the elements of $ \mathfrak{ n } $ given by
 \begin{equation*}\small
     \begin{split}
         \sec t \, E_{ 1, k }, \quad  \sec t \, F_{ 1, k }, \quad \csc t \, E_{ p + 2, \ell }, \quad \csc t \, F_{ p + 2, \ell }, \quad \tfrac{ 2 }{ \eta } \csc 2t \, D,
     \end{split}
 \end{equation*}
 are such that their corresponding action fields form an orthonormal basis of the tangent space of the orbit at $ \gamma ( t ) $ for $ t \neq 0, \tfrac{ \pi }{ 2 } $. By Proposition~\ref{Pt}, at the points where $ r ( t ) \neq \tfrac{ \pi }{ 2 } $ the endomorphisms $ P_{ r ( t ) } $ diagonalize simultaneously. Hence, we get the following:
 \begin{lemma}\label{lemmatangentialpart}
     Consider the natural $ \SU ( p + 1 ) \times \SU ( n - p ) $-action on $ \CPn $ with $ 0 \leq p < n $. The tangential component of the tension field of a $(k,r)$-map, where $ k \in 2 \Z + 1 $, vanishes.
 \end{lemma}

 \subsection{Normal component of the tension field}
 
 The normal component of the tension field for the $ ( k , r ) $-maps has been derived in~\cite{puttman2019harmonic}[Theorem 3.4], and can be written in terms of the endomorphisms $ P_t $ as
\begin{equation*}\small
    \tau^{ \nor }_{ \vert \gamma ( t ) } = \ddot r ( t ) + \tfrac{ 1 }{ 2 } \dot r ( t ) \tr P_t^{ - 1 } \dot P_t - \tfrac{ 1 }{ 2 } \tr P_t^{ - 1 } ( \dot P )_{ r ( t ) }
\end{equation*}
for every $ t \notin \tfrac{ \pi }{ 2 } \Z $. We obtain the following result using this identity, Proposition~\ref{Pt} and Lemma~\ref{lemmatangentialpart}.
\begin{theorem}
Consider the natural $ \SU ( p + 1 ) \times \SU ( n - p ) $-action on $ \CPn $ with $ 0 \leq p < n $. The tension field of a $ ( k, r ) $-map vanishes if and only if $ r $ satisfies the boundary value problem
\begin{equation}\label{ODE}
    \begin{split}
        \ddot{ r } ( t ) &+ [ ( 2 n - 2 p - 1 ) \cot t - ( 2 p + 1 ) \tan t ] \, \dot{ r } ( t ) \\
        &+ \left[ \frac{ p }{ \cos^2 t } - \frac{ ( n - p - 1 ) }{ \sin^2 t } \right] \, \sin 2 r ( t ) - \frac{ \sin 4 r ( t ) }{ \sin^2 2 t } = 0
    \end{split}
\end{equation}
for smooth functions $ r : ( 0, \frac{ \pi }{ 2 } ) \to \R $ with
\begin{equation}\label{boundaryconditions}
    \lim_{ t \rightarrow 0 } r ( t ) = 0 \quad \text{and} \quad \lim_{ t \rightarrow \frac{ \pi }{ 2 }} r ( t ) = k \tfrac{ \pi }{ 2 },
\end{equation}
where $ k \in 2 \Z + 1 $.
\end{theorem}

\section{New harmonic self-maps of \texorpdfstring{$ \CPn $}{CPn}}\label{sectionsolutions}

In this section, we prove the existence of infinitely many harmonic self-maps of $ \CPn $ and study their limiting configuration by giving a family of uncountable many explicit solutions to the boundary value problem~\eqref{ODE},~\eqref{boundaryconditions}. Moreover, we explicitly compute the energy for each of these harmonic maps.

We start by proving that every solution of~\eqref{ODE} we construct here is unique in a certain class of functions, as the following lemma states.
\begin{lemma}\label{uniqueness}
    Every smooth solution $ r : ( 0, \frac{ \pi }{ 2 } ) \to \R $ of~\eqref{ODE} with singular initial data
    \begin{equation*}
        \lim_{ t \rightarrow 0^{ + } } r ( t ) = 0 \quad \text{and} \quad \lim_{ t \rightarrow 0^{ + } } \dot{ r } ( t ) = \rho,
    \end{equation*}
    is unique.
\end{lemma}
\begin{proof}
    The strategy is similar to the approach used in~\cite{gastel2004harmonic}[Lemma~2.1]. The initial data implies that
    \begin{equation}\label{conditionuniqueness2}\small
        r ( t ) = \rho t + O ( t^2 ), \quad \dot{ r } ( t ) = \rho + O ( t ) \quad \text{as} \quad t \to 0^{ + }.
    \end{equation}
    If we perform the change of coordinates
    \begin{equation*}\small
        t ( x ) = \arctan e^x ,
    \end{equation*}
    equation~\eqref{ODE} reads
    \begin{equation}\label{odechangeofvariable}\small
        \begin{split}
            \ddot{ r } ( x ) & + \left[ ( n - p - 1 ) e^{ - x } - p e^x \right] \, \tfrac{2}{ e^x + e^{ - x } } \, \dot{ r } ( x ) \\
            & - \left[ ( n - p - 1 ) e^{ - x } - p e^x \right] \, \tfrac{ 1 }{ e^x + e^{ - x } } \, \sin 2 r ( x ) - \tfrac{ 1 }{ 4 } \sin 4 r ( x ) = 0 ,
        \end{split}
    \end{equation}
    and~\eqref{conditionuniqueness2} is translated to
    \begin{equation}\label{conditionuniquenesschangeofvariable}\small
        r ( x ) = \rho e^x + O ( e^{ 2 x } ), \quad \dot{ r } ( x ) = \rho e^x + O ( e^{ 2 x } ) \quad \text{as} \quad x \to - \infty.
    \end{equation}
    
    Suppose then that $ r_1 $ and $ r_2 $ are two solutions of~\eqref{odechangeofvariable} satisfying~\eqref{conditionuniquenesschangeofvariable}. For every $ s \in ( 0 , 1 ) $, define the functions $ \eta_s $ by
    \begin{equation*}\small
        \eta_s ( x ) := r_1 ( x ) + s ( e^x + e^{ 2 x } ).
    \end{equation*}
    Due to~\eqref{conditionuniquenesschangeofvariable}, there exists a constant $ c ( s ) \in \R $ such that
    \begin{equation}\label{inequalityuniqueness}\small
        \eta_{ - s } ( x ) \leq r_2 ( x ) \leq \eta_s ( x )
    \end{equation}
    for $ x \in ( - \infty , c ( s ) ) $. The rest of the proof is dedicated to showing that there exists a constant $ \kappa \in \R $ independent of $ s $ such that property~\eqref{inequalityuniqueness} can be extended to $ ( - \infty , \kappa ) $ for every $ s $. This would imply that $ r_2 $ lies between $ \eta_{ - s } $ and $ \eta_{ s } $ for $ s $ arbitrarily small, so $ r_1 = r_2 $ in $ ( - \infty , \kappa ) $ and therefore in all of $ \R $.
    
    Rewrite~\eqref{odechangeofvariable} as $ L r = 0 $, then
    \begin{equation*}\small
        L \eta_s ( x )  = s ( 2 n - 2 p + 1 ) e^{ 2 x } + s O ( e^{ 3 x } ) \quad \text{as} \quad x \to - \infty.
    \end{equation*}
    This implies the existence of a constant $ c_1 $ (independent of $ s $) such that
    \begin{equation}\label{c1}\small
        L ( \eta_{ - s } ) ( x ) < 0, \quad L ( \eta_s ) ( x ) > 0,
    \end{equation}
    for $ x \in ( - \infty , c_1 ) $. Note also that
    \begin{equation}\label{c2}\small
        f ( x ) = \left[ ( n - p - 1 ) e^{ - x } - p e^x \right] \tfrac{ 1 }{ e^x + e^{ - x } } > 0
    \end{equation}
    for $ x \in ( - \infty , c_2 ) $, where $ c_2 = \frac{ 1 }{ 2 } \ln \frac{ n - p - 1 }{ p } $. Furthermore, if we write $ { h_{ s }^+ := \eta_s - r_2 } $, then $ h_s^+ ( x ) = s e^x + O ( e^{ 2 x } ) $ as $ x \to - \infty $, so there exists $ c_3 $ such that
    \begin{equation}\label{c3}\small
        | h_s^+ ( x ) | \leq \frac{ \pi }{ 8 }
    \end{equation}
    holds for every $ x \in ( - \infty , c_3 ) $. The constant $ c_3 $ can be taken independently of $ s $ because $ s \in ( 0 , 1 ) $.
    
    Take then
    \begin{equation*}\small
        \kappa = \min \{ c_1 , c_2 , c_3 \}.
    \end{equation*}
    Arguing by contradiction, suppose there exists $ x_0 ( s ) < \kappa $ such that $ h_s^+ ( x_0 ( s ) ) = 0 $. Without loss of generality, we can assume that $ x_0 ( s ) $ is the first time this happens. Using~\eqref{c1} and since $ r_2 $ is a solution of~\eqref{odechangeofvariable},
    \begin{equation*}\small
        \begin{split}
            0 < L ( \eta_s ) = L ( \eta_s ) - L ( r_2 ) =& \ddot{ h }_s^+ ( x ) + 2 f ( x ) \, \dot{ h }_s^+ ( x ) - f ( x ) ( \sin 2 \eta_s ( x ) - \sin 2 r_2 ( x ) ) \\
            &- \tfrac{ 1 }{ 4 } ( \sin 4 \eta_s ( x ) - \sin 4 r_2 ( x ) ).
        \end{split}
    \end{equation*}
    By~\eqref{c3} we have that $ \sin 2 \eta_s ( x ) > \sin 2 r_2 ( x ) $ and $ \sin 4 \eta_s ( x ) > \sin 4 r_2 ( x ) $ for $ x \in ( -\infty, x_0 ( s ) ) $. This, together with~\eqref{c2}, yields the following differential inequality
    \begin{equation*}\small
        0 < \ddot{ h }_s^+ ( x ) + 2 f ( x ) \, \dot{ h }_s^+ ( x ) 
    \end{equation*}
    for every $ x \in ( - \infty , x_0 ( s ) ) $, which implies that $ h_s^+ $ cannot have a maximum there, contradicting our assumption. A similar argument applies to $ h_s^- := r_2 - \eta_{ - s } $. In other words, property~\eqref{inequalityuniqueness} holds in $ ( - \infty , \kappa ) $, as we wanted to show.
\end{proof}

\begin{theorem}\label{newhm}
    Let $ \rho \in \R $ and $ \ell \in \Z $, the functions defined by
    \begin{equation*}
        r_{ \rho, \ell } ( t ) = \arctan ( \rho \tan t ) + \ell \pi, \quad \kappa_{\ell} ( t ) = \ell \tfrac{ \pi }{ 2 } 
    \end{equation*}
    for every $ t \in ( 0, \tfrac{ \pi }{ 2 } ) $ satisfy the following properties:
    \begin{enumerate}
        \item[1.]  As $ \rho $ goes to $ \infty $, $ r_{ \rho, \ell }$ converges uniformly to $ \kappa_{ 2\ell + 1 } $. As $ \rho $ goes to $ - \infty $, $ r_{ \rho, \ell } $ converges uniformly to $ \kappa_{ 2\ell - 1 } $.
        \item[2.] The functions $ r_{ \rho, \ell } $ and $ \kappa_{ \rho, \ell } $ are solutions for the ordinary differential equation~\eqref{ODE}.
        \item[3.] If $ { \rho \neq 0} $, the function $ r_{ \rho, 0 } $ is the unique solution for~\eqref{ODE},~\eqref{boundaryconditions} satisfying $ \dot{ r } ( t ) \to \rho$ as $ t \to 0^+ $.
    \end{enumerate}
\end{theorem}
\begin{proof}
    Since the functions defined by $ \sin 2 r, \cos 2 r $ for every $ r \in \R $ are $ \pi $-periodic, if we show that $ r_{ \rho, 0 } $ is a solution for the ordinary differential equation~\eqref{ODE}, then so is $ r_{ \rho, \ell } $ for any $ \ell \in \Z $. To prove this, note the following relations:
    \begin{align*}\small
        \dot{ r }_{ \rho, 0 } ( t ) = \frac{ \rho }{ \rho^2 \sin^2 t + \cos^2 t }&,  && \hspace{0.67cm} \ddot{ r }_{ \rho,0 } ( t ) = \frac{ ( \rho - \rho^3 ) \sin 2 t }{ ( \rho^2 \sin^2 t + \cos^2 t )^2 }, \\[1ex]
        \sin 2 r_{ \rho, 0 } ( t ) = \frac{ \rho \sin 2 t }{ \rho^2 \sin^2 t + \cos^2 t }&,  &&\cos 2 r_{ \rho, 0 } ( t ) = \frac{ \cos^2 t - \rho^2 \sin^2 t }{ \rho^2 \sin^2 t + \cos^2 t } .
    \end{align*}
    Consequently, we obtain the following equality
    \begin{equation*}\small
        \begin{split}
            \frac{ 2 \rho \cot 2 t }{ \rho^2 \sin^2 t + \cos^2 t } &= [ ( 2 n - 2 p - 1 ) \cot t - ( 2 p + 1 ) \tan t ] \dot{r}_{ \rho, 0 } ( t ) \\[1ex]
            &+ \left[ \frac{ p } { \cos^2 t } - \frac{ ( n - p - 1 ) }{ \sin^2 t } \right] \sin 2 r_{ \rho, 0 } ( t ).
        \end{split}
    \end{equation*}
    If we plug this identity in equation~\eqref{ODE}, we get the expression
    \begin{equation*}\small
        \begin{split}
        &\ddot{ r }_{ \rho, 0 } ( t ) + \frac{ 2 \rho \cot 2 t }{ \rho^2 \sin^2 t + \cos^2 t } - 2 \frac{ \sin 2 r_{ \rho, 0 } ( t ) \, \cos 2 r_{ \rho, 0 } ( t ) }{ \sin^2 2 t } \\[1ex]
        =&\frac{ ( \rho - \rho^3 ) \sin 2 t }{ ( \rho^2 \sin^2 t + \cos^2 t )^2 } +  \frac{ 2 \rho \cot 2 t }{ \rho^2 \sin^2 t + \cos^2 t } - \frac{ \rho \cot t - \rho^3 \tan t }{ ( \rho^2 \sin^2 t +\cos^2 t )^2 }=0.
        \end{split}
    \end{equation*}
    Uniqueness is a consequence of Lemma~\ref{uniqueness}. The remaining properties follow directly. 
\end{proof}

For every $ \rho \in \R \setminus \{ 0 \} $, we will write $ \psi_{ \rho } $ for the harmonic self-map of $ \CPn $ corresponding to the solution $ r_{ \rho, 0 } $. This is, the $ ( \pm 1, r_{ \rho, 0 } ) $-map constructed in Theorem~\ref{newhm}.

One might ask whether the maps $ \psi_{ \rho } $ are holomorphic, antiholomorphic, or neither. It turns out that the answer depends on the value of $ \rho $. To check this, recall that the holomorphicity of $ \psi_{ \rho } $ relies on the condition
\begin{equation*}\small
    d \psi \circ J = J \circ d \psi.
\end{equation*}
Using Lemma~\ref{actionfieldsofbasis}, one can easily check that this condition is trivial for the vectors tangent to the orbits and that the only relevant condition is given by
\begin{equation*}\small
    d \psi \, ( J \dot{ \gamma } ( t ) ) = J d \psi \, ( \dot{ \gamma } ( t ) ) .
\end{equation*}
A straightforward computation shows that, in our case, this condition reads as the following first-order ordinary differential equation
\begin{equation*}\small
    \dot{ r } ( t ) \, ( 1 + \tan^2 r ( t ) ) \, \tan t = ( 1 + \tan^2 t ) \, \tan r ( t ) .
\end{equation*}
The solutions to this equation are precisely the functions $ r_{ \rho, 0 } $ given in Theorem~\ref{newhm} for $ \rho > 0 $. Hence, the maps $ \psi_{ \rho } $ are all holomorphic for $ \rho > 0 $ and non-holomorphic for $ \rho < 0 $. Note that the maps $ \psi_{ \rho } $ are not anti-holomorphic either, since the condition $ d \psi \circ J = - J \circ d \psi $ does not hold for the vectors tangent to the orbits. For the sake of completeness, in the next section we will prove that the maps $ \psi_{ \rho } $ are holomorphic for $ \rho > 0 $ using a known result concerning harmonic maps between compact K\"ahler manifolds.

One might also wonder how the energy of the harmonic maps $ \psi_{ \rho } $ depends on the parameter $ \rho $: it turns out that $ E ( \psi_{ \rho } ) $ is not only insensitive to $ \rho $ but also to the natural number $ p $, as the following proposition shows.
\begin{proposition}
    For $ \rho \neq 0 $, the energy of the harmonic map $ \psi_{ \rho } $ constructed above is given by
    \begin{equation*}\small
        E ( \psi_{ \rho } ) = n \text{Vol} ( \CPn ) = \tfrac{ \pi^n }{ ( n - 1 ) ! }.
    \end{equation*}
\end{proposition}
\begin{proof}
    A straightforward computation shows that the energy of a $(k,r)$-map is given by
    \begin{equation*}\small
        E ( \psi ) = \frac{ 1 }{ 2 } \int_{ \CPn } \left[ \dot{ r } ( t )^2 + 2 p \tfrac{ \cos^2 r ( t ) }{ \cos^2 t } + 2 ( n - p - 1 ) \tfrac{ \sin^2 r ( t ) }{ \sin^2 t } + \tfrac{ \sin^2 2 r ( t ) }{ \sin^2 2 t }  \right] \, dV_{ g_{ FS } } .
    \end{equation*}
    If we plug now the functions $ r_{ \rho , 0 } $ for $ \rho \neq 0 $ and simplify, we obtain
    \begin{equation}\label{energycalc}\small
        E ( \psi_{ \rho } ) = \frac{ 1 }{ 2 } \int_{ \CPn } \left[ \frac{ 2 \kappa_1 + 2 \kappa_2 \cos^2 t }{ ( \rho^2 \, \sin^2 t + \cos^2 t )^2 } \right] \, dV_{ g_{ FS } }
    \end{equation}
    where
    \begin{equation*}\small
        \begin{split}
            \kappa_1 &= \rho^{ 4 } n - \rho^{ 4 } + \rho^{ 2 } - \rho^{ 2 } ( \rho^{ 2 } - 1 ) p, \\[1ex]
            \kappa_2 &= \rho^{ 4 } - \rho^{ 2 } - \rho^{ 2 } ( \rho^{ 2 } - 1 ) n + ( \rho^{ 2 } - 1 )^2 p .
        \end{split}
    \end{equation*}
    From this, we see that $ E ( \psi_{ \rho } ) = E( \psi_{ - \rho } ) $. Since $ \{ \psi_{ \rho } \}_{ \rho > 0 } $ is a one-parameter family of harmonic maps, we have that
    \begin{equation*}\small
        \frac{ d }{ d \rho } E ( \psi_{ \rho } ) =  - \int_{ \CPn } g_{ FS } \left( \frac{ d }{ d\rho } \psi_{\rho }  , \tau( \psi_{ \rho } ) \right) \, dV_{ g_{ FS } } = 0.
    \end{equation*}
    In particular, it is enough to compute~\eqref{energycalc} for $ \rho = 1 $. Substituting, we get
    \begin{equation*}\small
        \begin{split}
            E ( \psi_{ \rho } ) &= n \int_{ \CPn } \, dV_{ g_{ FS } } = n \text{Vol} ( \CPn ) = \tfrac{ \pi^n }{ ( n - 1 ) ! } .
        \end{split}
    \end{equation*}
\end{proof}

\section{Stability of solutions}\label{sectionstability}

In the present section, we study the stability of the harmonic maps constructed. For every $ \rho > 0 $, we prove that the harmonic map $ \psi_{ \rho } $ is weakly stable. For $ { \rho < 0 } $, we show that $ \psi_{ \rho } $ is equivariantly weakly stable. Moreover, we explicitly compute the equivariant spectra of the maps $ \psi_1, \psi_{ - 1 } $ in the family of $ \SU ( \tfrac{ n + 1 }{ 2 } ) \times \SU ( \tfrac{ n + 1 }{ 2 } ) $-equivariant self-maps of $ \CPn $.

\begin{theorem}\label{stability1}
    For every $ \rho >0 $, the harmonic map $ \psi_{ \rho } : \CPn \to \CPn $ is weakly stable.
\end{theorem}
\begin{proof}
    The fact that $ \psi_{ \rho } $ is harmonic for every $ \rho > 0 $ follows from Theorem~\ref{newhm}. In particular, $ \psi_1 $ is the identity map of $ \CPn $, and hence, holomorphic.

    Moreover, $ \psi_{ \rho } : \CPn \to \CPn $ where $ \rho \in ( 0, \infty ) $ is a smooth variation of harmonic self-maps of a compact K\"ahler manifold. Thus, as a consequence of Theorem~\ref{regularity}, $ \psi_{ \rho } $ is a holomorphic harmonic map for every $ \rho \in ( 0, \infty ) $. The weak stability follows then as a direct application of Theorem~\ref{weaklystability}.
\end{proof}

Recall that the eigenvalue problem describing the equivariant stability of a harmonic $ ( k, r ) $-map in a cohomogeneity one manifold can be expressed as
\begin{equation*}\small
    \ddot \xi ( t ) + \tfrac{ 1 }{ 2 } \tr ( P_t^{ - 1 } \dot P_t ) \, \dot \xi ( t ) - \tfrac{ 1 }{ 2 } \tr ( P_t^{ - 1 } \ddot P_{ r ( t ) } ) \, \xi ( t ) + \lambda \xi ( t ) = 0
\end{equation*}
where $ \xi \in C_0^{ \infty } ( [ 0, \tfrac{ \pi }{ 2 } ] ) $. Proposition~\ref{Pt} shows that in our case this expression reduces to
\begin{equation}\label{spparticular}\small
    \begin{split}
        &\ddot{ \xi } ( t ) + \left[ ( 2 n - 2 p - 1 ) \cot t - ( 2 p + 1 ) \tan t  \right] \, \dot{ \xi } ( t ) \\[1ex]
        &- \left[ 2 ( n - p - 1 ) \tfrac{ \cos 2r_{ \rho } ( t ) }{ \sin^2 t } - 2 p \tfrac{ \cos 2r_{ \rho } ( t ) }{ \cos^2 t } + 4 \tfrac{ \cos 4 r_{ \rho } ( t ) }{ \sin^2 2t } \right] \xi ( t ) + \lambda \xi ( t ) = 0
    \end{split}
\end{equation}
where $ \xi \in C_0^{ \infty } ( [ 0, \tfrac{ \pi }{ 2 } ] ) $.

\begin{theorem}
    For every $ \rho < 0 $, the map $ \psi_{ \rho } : \CPn \to \CPn $ is equivariantly weakly stable.
\end{theorem}
\begin{proof}
    Since
    \begin{equation*}\small
        \begin{split}
            \cos ( 2 r_{ - \rho } ( t ) ) = \cos ( - 2 r_{ \rho } ( t ) ) = \cos ( 2 r_{ \rho } ( t ) ), \\
            \cos ( 4 r_{ - \rho } ( t ) ) = \cos ( - 4 r_{ \rho } ( t ) ) = \cos ( 4 r_{ \rho } ( t ) ),
        \end{split}
    \end{equation*}
    the equivariant spectrum of $ \psi_{ \rho } $ coincides with the equivariant spectrum of $ \psi_{ - \rho } $. By Theorem~\ref{stability1}, for $ \rho > 0 $ the map $ \psi_{ \rho } $ is weakly stable. In particular, it is equivariantly weakly stable, so the eigenvalues of the spectral problem~\eqref{spparticular} are all non-negative. 
\end{proof}

Now we proceed to explicitly compute the spectra of the maps $ \psi_1, \psi_{ - 1 } $ in the case $ p = \tfrac{ n - 1 }{ 2 } $. The reason for this specific choice is the appearance of the symmetry $ ( x, \xi ) \to ( - x, - \xi ) $ in the spectral problem, which significantly simplifies the computations. After the substitution
\begin{equation*}\small
    t ( x ) = \arctan e^x,
\end{equation*}
the problem~\eqref{spparticular} reads
\begin{equation}\label{auxeigenvalueproblem}\small
    \ddot{ \xi } ( x ) - ( n - 1 ) \tanh x \, \dot{ \xi } ( x ) - n \tanh^2 x \, \xi ( x ) + ( \tfrac{ \lambda }{ 4 } + 1 ) \tfrac{ 1 }{ \cosh^2 x } \, \xi ( x ) = 0
\end{equation}
for $ \xi \in C_0^{ \infty } ( \R ) $.

\begin{proposition}
    The spectral problem~\eqref{auxeigenvalueproblem} describing the equivariant stability of the maps $ \psi_1 $ and $ \psi_{ - 1 } $, is solved by 
    \begin{equation*}\small
        \xi_j ( x ) = \tfrac{ 1 }{ \cosh x } \, P_j^{ ( \frac{ n + 1 }{ 2 }, \frac{ n + 1 }{ 2 } ) }( \tanh x ), \quad \lambda_j = 4 j ( j + n + 2 )
    \end{equation*}
    for $ j\in \N $, where $P_j^{ ( \frac{ n + 1 }{ 2 }, \frac{ n + 1 }{ 2 } ) } $ are the so-called Jacobi polynomials.  
\end{proposition}
\begin{proof}
With the ansatz
\begin{equation*}\small
    \xi ( x ) = \tfrac{ 1 }{ \cosh x } \, f ( x ),
\end{equation*}
equation~\eqref{auxeigenvalueproblem} can be written as
\begin{equation*}\small
    f'' ( x ) - ( n + 1 ) \tanh x \, f' ( x ) + \tfrac{ \lambda }{ 4 } f ( x ) \, \sech^2 x = 0 .
\end{equation*}
In order to reduce this equation to the form~\eqref{jacobipolynomial}, we plug
\begin{equation*}\small
    f ( x ) = u ( \tanh x ),
\end{equation*}
and thus
\begin{equation*}\small
    ( 1 - \tanh^2 x ) \, u'' ( \tanh x ) - ( n + 3 ) \tanh x \, u' ( \tanh x ) + \tfrac{ \lambda }{ 4 } u ( \tanh x ) = 0 .
\end{equation*}
Then the eigenvalue problem~\eqref{auxeigenvalueproblem} is solved by
\begin{equation*}\small
    \xi_j ( x ) = \tfrac{ 1 }{ \cosh x } \, P_j^{ ( \frac{ n + 1 }{ 2 }, \frac{ n + 1 }{ 2 } ) } ( \tanh x ), \quad \lambda_j = 4 j ( j + n + 2 )
\end{equation*}
for $ j \in \N $.
\end{proof}

\begin{remark}
\begin{enumerate}

    \item[1.] Urakawa had already derived equation~\eqref{ODE} in~\cite{urakawa1993equivariant}. He proved the existence of a solution satisfying $ 0 \leq r \leq \tfrac{\pi}{2}$ and $ k = 1 $ using some tools from the calculus of variations. Moreover, in our opinion the hypotheses given in the paper were too restrictive: the relation between $ n $ and $ p $ is non-trivial, while the solutions constructed here are independent of $ n $ and $ p $, and the unique holomorphic harmonic self-map in his setup is the identity map.
    
    \item[2.] One might ask whether there is a geometric reason for the existence of uncountable many equivariant harmonic self-maps of $ \CPn $ in every dimension for the action considered. In contrast, in the case of rotationally symmetric harmonic self-maps of spheres one needs to impose the condition $ 3 \leq n \leq 6 $. Even in this situation, a numerical analysis suggests that there are considerably fewer solutions. 
    
    On the one hand, it is clear that the K\"ahler structure of $ \CPn $ represents an advantage for the existence of harmonic maps. In this situation, every holomorphic map is harmonic, so one can try to recover holomorphic harmonic maps by the condition $ d \psi \circ J = J \circ d \psi $. In our case, this condition reduces to a first-order ordinary differential equation.

    On the other hand, it is well known that curvature plays an important role in the existence of harmonic maps between Riemannian manifolds: Eells and Sampson~\cite{eells1964harmonic} proved that if all the sectional curvatures of the target manifold are non-positive, then there exists a harmonic representative in every homotopy class. The positive curved case seems to be exceptional, see for example~\cite{bizon1997harmonic} or~\cite{smith1975harmonic}. Recently, Branding and Siffert~\cite{branding2022infinite} constructed infinitely many harmonic maps between ellipsoids in all dimensions by making the correct deformation of these manifolds depending on the dimension. 
    
    One can argue that the fact that the sphere is a $ 1 $-pinched manifold is too restrictive for the existence of rotationally symmetric solutions in large dimensions, and this is why when one deforms an ellipsoid in such a way that the pinching constant is small enough, new solutions appear. This could be another argument for the $ \CPn $ case, since the sectional curvature in the complex projective space varies between $ 1 $ and $ 4 $ (i.e., it is a $ \tfrac{ 1 }{ 4 } $-pinched manifold).

    The relation between the pinching constant and harmonic maps is not new. Let $ M $ be a Riemannian manifold of dimension $ m \geq 3 $: it has been proved in~\cite{yanglian1989non} that if $ M $ has a pinching constant $ \delta $ big enough, then a harmonic map cannot be weakly stable unless it is constant. Actually, $ \delta $ depends on the dimension $ \delta \equiv \delta ( n ) $, and it was suggested in~\cite{howard1985nonexistence} that $ \delta ( n ) \to 1 $ as $ n \to \infty $.

\end{enumerate}
\end{remark}

\bibliographystyle{plain}
\bibliography{Bibliography.bib}

\end{document}